\documentclass[11pt]{amsart}

\usepackage[T1]{fontenc}
\usepackage{lmodern}
\usepackage{microtype}
\usepackage{amsmath,amssymb,amsthm,mathtools}
\usepackage{enumitem}
\usepackage[hidelinks]{hyperref}

\newtheorem{theorem}{Theorem}[section]
\newtheorem{proposition}[theorem]{Proposition}
\newtheorem{corollary}[theorem]{Corollary}
\newtheorem{lemma}[theorem]{Lemma}
\theoremstyle{remark}
\newtheorem{remark}[theorem]{Remark}
\newtheorem{question}[theorem]{Question}

\newcommand{\C}{\mathbb{C}}
\newcommand{\K}{\mathbb{K}}
\newcommand{\spn}{\operatorname{span}}

\title[Idempotent-free non-solvable evolution algebras]
{Idempotent-free non-solvable evolution algebras over $\C$}

\makeatletter
\def\author@andify{%
  \nxandlist{\unskip ,\penalty-1 \space\ignorespaces}%
    {\unskip,\space}%
    {\unskip,\penalty-2 \space}%
}
\makeatother

\author[X.-Y. Hu]{Xing-Yu Hu}
\address{School of Mathematics and Statistics, Hanjiang Normal University, Shiyan 442000, Hubei, China}
\email{huxingyu@hjnu.edu.cn}

\author[R. Wen]{Ran Wen\textsuperscript{*}}
\address{School of Mathematics and Physics, Xinjiang Hetian College, Hetian 848000, Xinjiang, China}
\email{1837496033@qq.com}
\thanks{\(*\)Corresponding author: Ran Wen.}



\hypersetup{
  pdftitle={Idempotent-free non-solvable evolution algebras over C},
  pdfauthor={Xing-Yu Hu, Ran Wen},
  pdfsubject={Evolution algebras, solvability and idempotents},
  pdfkeywords={evolution algebra, solvable algebra, idempotent, derived series}
}

\begin{document}

\begin{abstract}
A recent conjecture states that a finite-dimensional complex evolution algebra is solvable if and only if it has no non-zero idempotents. We exhibit a three-dimensional counterexample over $\C$ whose isomorphism class already appears in the classification of three-dimensional complex evolution algebras. Since the conjecture is known in dimensions one and two, this counterexample has the smallest possible dimension. The algebra is defined over every field of characteristic different from $2$ and is the exceptional member of a one-parameter family of pairwise non-isomorphic non-solvable evolution algebras whose idempotents are determined explicitly. For the exceptional parameter, the derived series stabilises at a non-zero two-dimensional subalgebra whose only idempotent is zero. Direct sums with zero algebras give complex counterexamples in every dimension at least three. We also prove the conjectured equivalence whenever the stable term of the derived series is an evolution algebra.
\end{abstract}

\maketitle
\enlargethispage{5pt}

\section{Introduction}

An evolution algebra over a field $\K$ is a commutative $\K$-algebra $E$, not necessarily associative, that admits a basis $B=\{e_1,\ldots,e_n\}$ such that
\[
 e_i e_j=0 \qquad (i\ne j).
\]
Such a basis is called a natural basis \cite[p.~112]{TianVojtechovsky2006}. If
\[
 e_i^2=\sum_{j=1}^n a_{ij}e_j,
\]
then $M_B(E)=(a_{ij})$ denotes the structure matrix relative to $B$ \cite[p.~600]{GarciaPerez2026}. Evolution algebras were introduced in connection with non-Mendelian genetics and have since been studied as a class of non-associative algebras \cite{TianVojtechovsky2006,Tian2008}. Subsequent work has developed their nilpotency and solvability theory \cite{CamachoEtAl2013,CasasEtAl2014,FernandezEtAl2022,LadraEtAl2025}, graph-theoretic methods \cite{ElduqueLabra2015,CeballosNunezTenorio2022}, and the structure of evolution algebras satisfying $E^2=E$ \cite{SriwongsaZou2022,WeiZou2023,GarciaPerez2026}.

For subspaces $U,V$ of an algebra, write $UV$ for the span of the products $uv$ with $u\in U$ and $v\in V$. Define the derived series by
\[
 A^{(1)}=A, \qquad A^{(m+1)}=A^{(m)}A^{(m)} \quad (m\ge 1).
\]
The algebra $A$ is solvable if $A^{(m)}=0$ for some $m$ \cite[p.~602]{GarciaPerez2026}. An evolution algebra $E$ is regular (or perfect) if $E^2=E$ \cite[p.~600]{GarciaPerez2026}. An idempotent is an element $x\in A$ satisfying $x^2=x$ \cite[p.~601]{GarciaPerez2026}. We call $A$ idempotent-free if $0$ is its only idempotent.

Garc\'ia-Mart\'inez and P\'erez-Rodr\'iguez proved that every finite-dimensional complex regular evolution algebra admits a non-zero idempotent \cite[Theorem~3.5]{GarciaPerez2026}. They then conjectured that, for a finite-dimensional complex evolution algebra $E$,
\begin{equation}\label{eq:conjecture}
 E\text{ is solvable}
 \quad\Longleftrightarrow\quad
 E\text{ has no non-zero idempotents}.
\end{equation}
They verified \eqref{eq:conjecture} in dimensions one and two \cite[p.~602, following Conjecture~3.6]{GarciaPerez2026}.

A related pair of claims appears in \cite[Proposition~3 and Corollary~2]{Ceballos2026}, where every $S$-graphicable algebra is said to have only the zero idempotent and to be non-solvable. The first claim fails in general. For the $S$-graphicable algebra over $\C$ associated with $K_2$,
\[
 e_1^2=e_2, \qquad e_2^2=e_1,
\]
the element $e_1+e_2$ is a non-zero idempotent. Thus these claims do not establish a counterexample to \eqref{eq:conjecture}.

We construct a one-parameter family of non-solvable evolution algebras and determine the idempotents of every member. Proposition~\ref{prop:family-isomorphisms} shows that the squares of distinct members are non-isomorphic, and hence so are the members themselves. Its exceptional member has integral structure constants, is idempotent-free, and gives a three-dimensional counterexample over every field of characteristic different from $2$. Over $\C$, no counterexample exists in smaller dimension.

The classification in \cite[Theorem~3.5]{CabreraEtAl2017} covers all three-dimensional complex evolution algebras. Over $\C$, the isomorphism class of $E_2$ is already listed there. By \cite[Proposition~9]{CeballosNunezTenorio2022}, only four connected three-vertex pseudodigraph configurations remain potentially solvable. For each complex member of the family in Theorem~\ref{thm:family}, the pseudodigraph associated with the displayed natural basis lies outside those four configurations. The algebra in \cite[Example~1.2]{TianVojtechovsky2006} is isomorphic to $E_1$. Remark~\ref{rem:classification-location} records the explicit isomorphism, the position of $E_2$ in the classification, and the difference between the structure-matrix conventions.

Theorem~\ref{thm:family} determines the idempotents of every member of the family. At $\lambda=2$, the algebra $E_2$ has only the zero idempotent and therefore disproves \eqref{eq:conjecture}. Its stable term of the derived series is not an evolution algebra, and Proposition~\ref{prop:stable-term} shows that every finite-dimensional complex counterexample must have this property.

\section{Main results}

\subsection{A non-solvable family and its idempotents}

\begin{theorem}\label{thm:family}
Let $\K$ be a field, let $\lambda\in\K^\times$, and let $E_\lambda$ be the evolution algebra with natural basis $B=\{e_1,e_2,e_3\}$ and multiplication
\begin{equation}\label{eq:family-multiplication}
 e_1^2=-e_1-e_3, \qquad
 e_2^2=-2e_1+\lambda e_2+(\lambda-2)e_3, \qquad
 e_3^2=e_1+e_3.
\end{equation}
Then $E_\lambda$ is not solvable. If $\lambda\ne2$, it has exactly one non-zero idempotent, namely
\begin{equation}\label{eq:family-idempotent}
 \frac{1}{\lambda(2-\lambda)}(e_1+e_3)
 +\frac{1}{\lambda}(e_2+e_3).
\end{equation}
If $\lambda=2$, then $0$ is its only idempotent.
\end{theorem}

\begin{proof}
Put
\[
 u=e_1+e_3, \qquad v=e_2+e_3,
\]
and let $S=\spn\{u,v\}$. From \eqref{eq:family-multiplication},
\[
 e_1^2=-u, \qquad e_2^2=-2u+\lambda v, \qquad e_3^2=u.
\]
Since $\lambda\ne0$, we have $E_\lambda^2=S$. The products in $S$ are
\begin{equation}\label{eq:stable-products-family}
 u^2=0, \qquad uv=u, \qquad v^2=-u+\lambda v.
\end{equation}
Thus $S^2=S$, because $u=uv$ and $v=(u+v^2)/\lambda$. It follows that
\[
 E_\lambda^{(m)}=S \qquad (m\ge2),
\]
so $E_\lambda$ is not solvable.

Let $x\in E_\lambda$ be an idempotent. Since $x=x^2$, we have $x\in E_\lambda^2=S$. Write $x=au+bv$. By \eqref{eq:stable-products-family}, the equation $x^2=x$ is equivalent to
\[
 2ab-b^2=a, \qquad \lambda b^2=b.
\]
The second equation implies $b=0$ or $b=1/\lambda$. If $b=0$, then the first equation gives $a=0$. If $b=1/\lambda$ and $\lambda\ne2$, then
\[
 a=\frac{1}{\lambda(2-\lambda)},
\]
so $x$ is the idempotent displayed in \eqref{eq:family-idempotent}. If $\lambda=2$, then $\lambda\in\K^\times$ forces $\operatorname{char}\K\ne2$, and the first equation becomes
\[
 a-\frac14=a,
\]
which is impossible. The theorem follows.
\end{proof}

Specialising to $\lambda=2$ gives the counterexample.

\begin{corollary}\label{cor:counterexample}
Let $\K$ be a field of characteristic different from $2$, and let $E=E_2$. Then $E$ is a three-dimensional non-solvable evolution algebra whose only idempotent is $0$.
\end{corollary}

\begin{proof}
This is Theorem~\ref{thm:family} with $\lambda=2$.
\end{proof}

\begin{lemma}\label{lem:not-evolution}
Let $\K$ be a field, let $\lambda\in\K^\times$, and let $E_\lambda$ be as in Theorem~\ref{thm:family}. Set $S_\lambda=E_\lambda^2=\spn\{e_1+e_3,e_2+e_3\}$. Then $S_\lambda$ is not an evolution algebra.
\end{lemma}

\begin{proof}
Write $u=e_1+e_3$ and $v=e_2+e_3$. Suppose that $S_\lambda$ has a natural basis
\[
 f=au+bv, \qquad g=cu+dv.
\]
Then $ad-bc\ne0$ and $fg=0$. Since
\[
 u^2=0, \qquad uv=u, \qquad v^2=-u+\lambda v,
\]
we have
\[
 fg=(ad+bc-bd)u+\lambda bdv.
\]
Because $\lambda\ne0$, the equality $fg=0$ implies $bd=0$. If $b=0$, then $ad\ne0$, but the coefficient of $u$ in $fg$ is $ad$, a contradiction. If $d=0$, then $bc\ne0$, but the coefficient of $u$ is $bc$, again a contradiction.
\end{proof}

\begin{proposition}\label{prop:family-isomorphisms}
Let $\K$ be a field and let $\lambda,\mu\in\K^\times$. Then $E_\lambda^2$ and $E_\mu^2$ are isomorphic as $\K$-algebras if and only if $\lambda=\mu$. Consequently, $E_\lambda$ and $E_\mu$ are isomorphic as $\K$-algebras if and only if $\lambda=\mu$.
\end{proposition}

\begin{proof}
Only the forward implication in the first assertion requires proof. For $\nu\in\{\lambda,\mu\}$, write $S_\nu=E_\nu^2$ and choose $u_\nu,v_\nu$ as in the proof of Theorem~\ref{thm:family}. Thus
\[
 S_\nu=\spn\{u_\nu,v_\nu\}, \qquad
 u_\nu^2=0, \qquad u_\nu v_\nu=u_\nu, \qquad v_\nu^2=-u_\nu+\nu v_\nu.
\]
If $x=au_\nu+bv_\nu\in S_\nu$ and $x^2=0$, then the coefficient of $v_\nu$ in $x^2$ is $\nu b^2$. Since $\nu\ne0$, this gives $b=0$. Hence the square-zero elements of $S_\nu$ form the line $\K u_\nu$.

Suppose that $\varphi\colon S_\lambda\to S_\mu$ is an isomorphism. There are $\alpha,\beta,\gamma\in\K$ with $\alpha\gamma\ne0$ such that
\[
 \varphi(u_\lambda)=\alpha u_\mu, \qquad
 \varphi(v_\lambda)=\beta u_\mu+\gamma v_\mu.
\]
Applying $\varphi$ to $u_\lambda v_\lambda=u_\lambda$ gives $\gamma=1$. Comparing the coefficients of $v_\mu$ in
\[
 \varphi(v_\lambda)^2=(2\beta-1)u_\mu+\mu v_\mu
 \qquad\text{and}\qquad
 \varphi(v_\lambda^2)=(-\alpha+\lambda\beta)u_\mu+\lambda v_\mu
\]
then gives $\lambda=\mu$. The converse is immediate. Finally, any isomorphism $E_\lambda\to E_\mu$ maps $E_\lambda^2$ onto $E_\mu^2$, so the last assertion follows.
\end{proof}

\begin{remark}\label{rem:classification-location}
For $E=E_2$, the structure matrix with respect to the displayed natural basis is
\[
 M_B(E)=
 \begin{pmatrix}
 -1&0&-1\\
 -2&2&0\\
 1&0&1
 \end{pmatrix}.
\]
This matrix has rank two. More generally, the structure matrix of every finite-dimensional complex counterexample to \eqref{eq:conjecture} must be singular, since every regular complex evolution algebra has a non-zero idempotent by \cite[Theorem~3.5]{GarciaPerez2026}.

Over $\C$, put
\[
 f_1=\frac12e_2,\qquad f_2=e_3,\qquad f_3=\mathrm{i}e_1.
\]
This is again a natural basis, and a direct calculation shows that
\[
 f_1^2=f_1+\frac{\mathrm{i}}2f_3,\qquad
 f_2^2=f_2-\mathrm{i}f_3,\qquad
 f_3^2=f_2-\mathrm{i}f_3.
\]
Consequently,
\[
 M_{\{f_1,f_2,f_3\}}(E)=
 \begin{pmatrix}
 1&0&\mathrm{i}/2\\
 0&1&-\mathrm{i}\\
 0&1&-\mathrm{i}
 \end{pmatrix}.
\]
In \cite[p.~73]{CabreraEtAl2017}, structure matrices are transposed relative to our convention, so the tables in that paper must be compared with the transpose of the displayed matrix. The case $\dim E^2=2$ in \cite[Theorem~3.5(iii)]{CabreraEtAl2017} already contains the isomorphism class of the counterexample. For the representative $E_2$, Theorem~\ref{thm:family} shows that $E_2^{(m)}=E_2^2\ne0$ for every $m\ge2$ and that $0$ is its only idempotent. Hence $E_2$ disproves \eqref{eq:conjecture}.

Over $\C$, the algebra $A$ in \cite[Example~1.2]{TianVojtechovsky2006} is isomorphic to $E_1$ via $a_1\mapsto e_1+2e_3$, $a_2\mapsto2e_1+e_3$, and $a_3\mapsto e_2$. In the notation of that example, $A^2=\spn\{a_1+a_2,a_1+a_3\}$, and the isomorphism maps $A^2$ onto $E_1^2$. Thus $E_1^2$ already fails to be an evolution algebra, whereas the idempotent-free counterexample occurs at $\lambda=2$.
\end{remark}

\subsection{Minimal dimension and counterexamples in all dimensions}

\begin{corollary}\label{cor:all-dimensions}
Over $\C$, the equivalence between solvability and the absence of non-zero idempotents holds in dimensions at most two and fails in every dimension at least three.
\end{corollary}

\begin{proof}
In any algebra, solvability implies the absence of non-zero idempotents, since every non-zero idempotent belongs to every term of the derived series. Conversely, in dimension one write $e^2=\lambda e$. If $\lambda=0$, the algebra is solvable. If $\lambda\ne0$, then $\lambda^{-1}e$ is a non-zero idempotent.

Now let $E$ be two-dimensional and assume that it has no non-zero idempotents. If its structure matrix had full rank, then $E^2=E$, making $E$ regular. It would then have a non-zero idempotent by \cite[Theorem~3.5]{GarciaPerez2026}, a contradiction. Thus $\dim E^2\le1$. If $E^2=0$, then $E$ is solvable. Otherwise, write $E^2=\C u$. Since $u^2\in E^2$, we have $u^2=\lambda u$ for some $\lambda\in\C$. If $\lambda\ne0$, then $\lambda^{-1}u$ is a non-zero idempotent. Hence $u^2=0$, and therefore $E^{(3)}=0$. The equivalence follows in dimensions at most two, while Corollary~\ref{cor:counterexample} supplies a counterexample in dimension three.

For $n>3$, let $F=E_2$ be the three-dimensional algebra in Corollary~\ref{cor:counterexample}, let $Z_{n-3}$ denote the $(n-3)$-dimensional zero algebra, and put
\[
 E_n=F\oplus Z_{n-3}.
\]
Combining a natural basis of $F$ with one of $Z_{n-3}$ gives a natural basis of $E_n$, so $E_n$ is an evolution algebra. The proof of Theorem~\ref{thm:family} gives $F^{(m)}=F^2\ne0$ for every $m\ge2$. Hence $E_n^{(m)}=F^2\oplus0\ne0$ for every $m\ge2$, so $E_n$ is not solvable. If $(x,z)\in E_n$ is idempotent, then $x^2=x$ and $z^2=z$. The multiplication on $Z_{n-3}$ is zero, so $z=0$, and Corollary~\ref{cor:counterexample} then gives $x=0$.
\end{proof}

\begin{remark}
The minimal-dimension statement in Corollary~\ref{cor:all-dimensions} is specific to $\C$. Over $\mathbb F_3$, the two-dimensional simple evolution algebra $E$ given by
\[
 e_1^2=e_1+2e_2,\qquad e_2^2=e_1+e_2
\]
has no non-zero idempotents \cite[Example~44]{GonzalezNevado2025}. Since $E^2$ is a non-zero ideal and $E$ is simple, we have $E^2=E$, so the algebra is not solvable.
\end{remark}

\subsection{Stable derived terms and solvability}

A sufficient condition for the conjectured equivalence is that the stable term of the derived series be an evolution algebra.

\begin{proposition}\label{prop:stable-term}
Let $E$ be a finite-dimensional complex evolution algebra. The derived series of $E$ stabilises. Let $D$ denote its stable term. If $D$ is an evolution algebra under the induced multiplication, then the following conditions are equivalent.
\begin{enumerate}[label=\textup{(\roman*)}]
 \item $E$ is solvable.
 \item $E$ has no non-zero idempotents.
\end{enumerate}
\end{proposition}

\begin{proof}
By induction, every derived term is a subalgebra and the derived series is descending. Indeed, $E^{(1)}=E$ is a subalgebra. If $E^{(r)}$ is a subalgebra, then
\[
 E^{(r+1)}=E^{(r)}E^{(r)}\subseteq E^{(r)}.
\]
Since $E^{(r+1)}\subseteq E^{(r)}$, products of elements of $E^{(r+1)}$ lie in $E^{(r)}E^{(r)}=E^{(r+1)}$. Hence $E^{(r+1)}$ is also a subalgebra. Finite-dimensionality therefore gives an index $m$ such that
\[
 D=E^{(m)}=E^{(m+1)}.
\]

Every idempotent belongs to every derived term by induction from $x=x^2$, so solvability excludes non-zero idempotents. Conversely, if $E$ is not solvable, then $D\ne0$ and $D^2=D$. By hypothesis, $D$ is an evolution algebra and hence a regular complex evolution algebra. By \cite[Theorem~3.5]{GarciaPerez2026}, it contains a non-zero idempotent, which is also an idempotent of $E$.
\end{proof}

\begin{corollary}
Let $E$ be a finite-dimensional complex evolution algebra with natural basis $B$, and define the evolution operator $L_B\colon E\to E$ by
\[
 L_B(e_i)=e_i^2
\]
for every $e_i\in B$ \cite[Definition~5, p.~29]{Tian2008}. If $L_B$ is an algebra endomorphism, then $E$ is solvable if and only if it has no non-zero idempotents.
\end{corollary}

\begin{proof}
Under this hypothesis, every term of the derived series is an evolution subalgebra by \cite[Proposition~3.1]{FernandezEtAl2022}. In particular, the stable term $D$ in Proposition~\ref{prop:stable-term} is an evolution algebra, and the conclusion follows.
\end{proof}

\begin{remark}
For the algebra in Corollary~\ref{cor:counterexample}, the stable term $D=E^2$ satisfies $D^2=D$ and has no non-zero idempotents, but Lemma~\ref{lem:not-evolution} shows that $D$ is not an evolution algebra. Hence Proposition~\ref{prop:stable-term} does not apply. The same lemma shows that $E_\lambda^2$ is not an evolution algebra for any member of the family, although Theorem~\ref{thm:family} gives a non-zero idempotent whenever $\lambda\ne2$. The evolution operator $L_B$ for the displayed natural basis $B=\{e_1,e_2,e_3\}$ is not an algebra endomorphism either, since $L_B(e_2e_3)=0$ whereas $L_B(e_2)L_B(e_3)=2(e_1+e_3)\ne0$.
\end{remark}

Related work treats subalgebras with natural bases extendible to the ambient algebra \cite{CamachoKhudoyberdiyevOmirov2019}, evolution subalgebras generated by idempotents \cite{MukhamedovQaralleh2023}, and closure under subalgebras in the regular case \cite{LadraPerez2025}. The counterexample and Proposition~\ref{prop:stable-term} lead to the following more specific question.

\begin{question}\label{question:stable-term}
For which finite-dimensional evolution algebras is the stable term of the derived series an evolution algebra?
\end{question}

\end{document}